\documentclass{amsproc}

\usepackage{amsfonts,amssymb,amsmath,amsthm}
\usepackage{url, amsmath}
\usepackage{enumerate}
\usepackage[utf8]{inputenc}
\usepackage[T1]{fontenc}
\usepackage{graphicx}
\usepackage{epsf, subfigure, verbatim}
\usepackage{epsfig}
\usepackage{enumerate}

\newcommand{\closure}[2][3]{%
  {}\mkern#1mu\overline{\mkern-#1mu#2}}

\newtheorem{theorem}{Theorem}[section]
\newtheorem{lemma}[theorem]{Lemma}
\newtheorem{proposition}[theorem]{Proposition}
\newtheorem{corollary}[theorem]{Corollary}

\theoremstyle{definition}

\theoremstyle{remark}
\newtheorem{remark}[theorem]{Remark}

\numberwithin{equation}{section}

\begin{document}

\title[Bounds for the $p$-Laplacian]{Bounds on the Principal Frequency of the $p$-Laplacian}


\author{Guillaume Poliquin}
\address{Département de mathématiques et de statistique \\
Université de Montréal, CP 6128 succ. Centre-Ville, Montréal\\
H3C 3J7, Canada.}
\email{gpoliquin@dms.umontreal.ca}
\thanks{Research supported by a NSERC scholarship}



\date{}


\begin{abstract}

This paper is concerned with the lower bounds for the principal frequency of the $p$-Laplacian on $n$-dimensional Euclidean domains. In particular, we extend the classical results involving the inner radius of a domain and the first eigenvalue of the Laplace operator to the case $p\neq2$. As a by-product, we obtain  a lower bound on the size of the nodal set of an eigenfunction of the $p$-Laplacian on planar domains.

\end{abstract}
\maketitle

\section{Overview of the $p$-Laplacian}

\subsection{Physical models involving the $p$-Laplacian}

Let $\Omega$ be a bounded open subset of $\mathbb{R}^n$. For $1 < p < \infty$, the $p-$Laplacian of a function $f$ on $\Omega$ is defined by $\Delta_p f =  \operatorname{div}(|\nabla f|^{p-2} \nabla f)$ for suitable $f$. The $p$-Laplacian can be used to model the flow of a fluid through porous media in turbulent regime (see for instance \cite{DHT, DTh}) or the glacier's ice when treated as a non-Newtonian fluid with a nonlinear relationship between the rate deformation tensor and the deviatoric stress tensor (see \cite{GR}). It is also used in the Hele-Shaw approximation, a moving boundary problem (see \cite{KMC}).

Let us present a model, well known in the case of the Laplace operator, which remains very useful to understand the physical meaning behind some inequalities that we shall prove, in particular those involving the inner radius of $\Omega$. The nonlinearity of the $p$-Laplacian is often used to reflect the impact of non ideal material to the usual vibrating homogeneous elastic membrane, modeled by the Laplace operator. Thus, the following is used to describe a nonlinear elastic membrane under the load $f$,
\begin{eqnarray}
-\Delta_p (u) = f \qquad\mbox{ in } \Omega, \\
u=0 \quad\qquad \mbox{ on } \partial\Omega. \nonumber
\end{eqnarray}
The solution $u$ stands for the deformation of the membrane from the rest position (see \cite{CEP, Si}). In that case, its deformation energy is given by $\int_\Omega |\nabla u|^p dx$. Therefore, a minimizer of the Rayleigh quotient,
$$ \frac{\int_\Omega |\nabla u|^p dx}{\int_\Omega |u|^p dx},$$
on $W_0^{1,p}(\Omega)$ satisfies $-\Delta_p (u) =  \lambda_{1,p}|u|^{p-2}u  \mbox{ in } \Omega.$ Here $\lambda_{1,p}$ is usually referred as the principal frequency of the vibrating non elastic membrane.

\subsection{The eigenvalue problem for the $p$-Laplacian}

For $1 < p < \infty$, we study the following eigenvalue problem:

\begin{equation}\label{pLaplacian}
  \Delta_p u + \lambda |u|^{p-2}u=0 \mbox{ in } \Omega,
\end{equation}
where we impose the Dirichlet boundary condition and consider $\lambda$ to be the real spectral parameter. We say that $\lambda$ is an eigenvalue of $-\Delta_p$ if \eqref{pLaplacian} has a nontrivial weak solution $u_\lambda \in W^{1,p}_0(\Omega)$. That is, for any $v \in  C^\infty_0(\Omega)$,
\begin{equation}
\int_\Omega |\nabla u_\lambda|^{p-2} \nabla u_\lambda \cdot \nabla v - \lambda \int_\Omega |u_\lambda|^{p-2} u_\lambda v=0.
\end{equation}
 The function $u_\lambda$ is then called an eigenfunction of $-\Delta_p$ associated to the eigenvalue $\lambda$.

%
If $n\geq2$ and $p=2$, it is well known that one can obtain an increasing sequence of eigenvalues tending to $+\infty$ via the Rayleigh-Ritz method. Moreover, those are all the eigenvalues of $\Delta_2$. Linearity is a crucial component in the argument.

For the general case, it is known that the first eigenvalue of the Dirichlet eigenvalue problem of the $p$-Laplace operator, denoted by $\lambda_{1,p}$, is characterized as,
\begin{equation}\label{var1}
\lambda_{1,p} = \min_{0 \neq u\in C^\infty_0(\Omega)} \left \{ \frac{ \int_\Omega |\nabla u|^p dx}{\int_\Omega |u|^p dx} \right \}.
\end{equation}
The infinimum is attained for a function $u_{1,p} \in W^{1,p}_0(\Omega)$. In addition, $\lambda_{1,p}$ is simple and isolated (there is no sequence of eigenvalues such that $\lambda_{k,p}$ tends to $\lambda_{1,p}$; see \cite{Lind2}). Moreover, the eigenfunction $u_{1}$  associated to $\lambda_{1,p}$ does not change sign, and it is the only such eigenfunction (a proof can be found in \cite{Lind2}). As for $\lambda_{2,p} > \lambda_{1,p}$, it allows a min-max characterization and every eigenfunction associated to $\lambda_{2,p}$ changes sign only once in $\Omega$ (it was first shown in \cite{ACFK}; see also \cite{D1}). It is not known if $\lambda_{2,p}$ is isolated. Via for instance Lyusternick-Schnirelmann maximum principle,  it is possible to construct $\lambda_{k,p}$ for $k\geq 3$ and hence obtain an increasing sequence of eigenvalues of \eqref{pLaplacian}. There exist other variational characterizations of the $\Delta_p$ eigenvalues. However, no matter what variational characterization one chooses, it always remains to show that all the eigenvalues obtained exhaust the whole spectrum of $\Delta_p$.



\section{Introduction and main results}

\subsection{The principal frequency and the inradius}

Using the domain monotoni\-city property, it is easy to obtain an upper bound for the principal frequency of the $p$-Laplacian. Indeed, for $B_r \subset \Omega$, we have that
$
\lambda_{1,p}(\Omega)\leq \lambda_{1,p}(B_r) = \frac{\lambda_{1,p}(B_1)}{r^p}.
$
Therefore, if we consider the largest ball that can be inscribed in $\Omega$, we get
\begin{equation}\label{monotonicite}
\rho_\Omega \leq \left( \frac{\lambda_{1,p}(B_1)}{\lambda_{1,p}(\Omega)} \right)^{\frac{1}{p}},
\end{equation}
where $\rho_\Omega:=\sup\{r:\exists B_r \subset \Omega\}$. Note that unlike the case $p=2$ corresponding to the Laplace operator, there are no explicit formulas for $\lambda_{1,p}$ of a ball.

Lower bounds involving the principal frequency are a greater challenge. Nevertheless, some results are known. Let us start by recalling that the classical Faber-Krahn inequality can be adapted to the $p$-Laplacian as noted in \cite[p. 224]{Lind1} and \cite[p. 3353]{Hua} (see also the rearrangement results in \cite{Ka}) : among all domains of given $n$-dimensional volume, the ball minimizes $\lambda_{1,p}$. In other words, we have that $\lambda_{1,p}(\Omega) \geq \lambda_{1,p}(\Omega^*),$ where $\Omega^*$ stands for the $n$-dimensional ball of same volume than $\Omega$.   

For the Laplacian, lower bounds of the type \begin{equation}\label{goal} \lambda_{1,2}(\Omega) \geq \alpha_{n,2} \ \rho_{\Omega}^{-2},\end{equation}
where $\alpha_{n,2} >0$ is a positive constant, have been studied extensively. If $n=2$, the first result proved in that direction is due to J. Hersch in \cite{He}, and states that for convex simply connected planar domains, the latter inequality holds with the constant $\alpha_{2,2} = \frac{\pi^2}{4}$.
This result was later improved by E. Makai in \cite{Mak}. For all simply connected domains, he obtained the constant $\alpha_{2,2} = \frac{1}{4}$.
An adaptation of Makai's method for the pseudo $p$-Laplacian was studied in \cite{B} and lead to a similar lower bound for simply-connected convex planar domains.

By a different approach, W. K. Hayman also obtained a bound for simply-con\-nected domains with the constant, $\alpha_{2,2}= 1/900$. R. Osserman (see \cite{O1, O2, O3}) later improved that result to $\alpha_{2,2} =1/4$. R. Osserman also relaxed the assumption of simple connectedness of the domain by considering the connectivity $k$ of a planar domain has a parameter. He obtained a similar lower bound for domains of connectivity $k \geq 2$, a result that was improved in \cite{Cr}.


In higher dimensions, it was first noted by W. K. Hayman in \cite{Hay} that if $A$ is a ball with many narrow inward pointing spikes removed from it, then $\lambda_{1,2}(A) = \lambda_{1,2}(\mbox{Ball}),$ but in that case the inradius of $A$  tends to $0$. Therefore,  bounds of the type,
$$\lambda_{1,2}(\Omega) \geq \alpha_{n,2} \ \rho_{\Omega}^{-2},$$
are generally not possible to obtain even if $\Omega$ is assumed to be simply connected. The higher dimensional case for Euclidean domains is discussed in Section \ref{S3}. Similar bounds on manifolds are presented in Section \ref{S4}.

In the next subsection, we extend some of these results to the case $p\neq 2$. The size nodal set of an eigenfunction of the $p$-Laplacian is also discussed. All proofs are given in Section \ref{S5}.



\subsection{First eigenvalue of the $p$-Laplacian and inradius of planar domains}

We present some lower bounds for the principal frequency involving the inradius of a planar domain with an explicit constant,
 $$\lambda_{1,p}(\Omega) \geq \alpha_{2,p} \ \rho_{\Omega}^{-p}.$$
We do so by adapting proofs obtained for the usual Laplacian. We need two main ingredients:  a modified Cheeger-type inequality (see the original result for the Laplacian in \cite{Che}; a generalized version for the $p$-Laplacian can be found in \cite{KF}), and a geometric inequality relating the ratio of the length of the boundary of a domain and its area. Regarding the modified Cheeger inequality, it consists of an adapted version of a result in \cite{O4}:

\begin{lemma}\label{CheegerLemma}
Let $(S,g)$ be a Riemannian surface, and let $D \subset S$ be a domain homeomorphic to a planar domain of finite connectivity $k$. Let $F_k$ be the family of relatively compact subdomains of $D$ with smooth boundary and with connectivity at most $k$. Let
\begin{equation}\label{chegconst}
h_k(D)=\inf_{D'\in F_k} \frac{|\partial D'|}{|D'|},
\end{equation}
where $|D|$ is the area of $D'$ and $|\partial D'|$ is the length of its boundary. Then,
\begin{equation}\label{14}  \lambda_{1,p}(D) \geq \left(\frac{h_k(D)}{p} \right)^{p}. \end{equation}
\end{lemma}
The crucial point of Lemma \ref{CheegerLemma} resides in the fact that the Cheeger constant is computed among subdomains that have a connectivity of at most the connectivity of the domain, which allows us to use geometric inequalities accordingly.

We start by proving the following extension of Osserman-Croke's result to the $p$-Laplacian. We actually generalize a stronger result that was implicit in \cite{Mak} and \cite{O1}, but made explicit in \cite{G}. Instead of considering the inner radius, we use the reduced inradius, which is defined by:
$$ \tilde{\rho}_\Omega := \frac{\rho_{\Omega}}{1+\frac{\pi \rho_\Omega^2}{|\Omega|}}.$$
Notice that $\frac{\rho_\Omega }{2} < \tilde{\rho}_\Omega< \rho_\Omega $.

The first main result consists of extending classical planar inradius bounds of the Laplace operator to the case of the $p$-Laplacian:
\begin{theorem}\label{OssCroP1}
Let $\Omega$ be a domain in $\mathbb R^2$. If $\Omega$ is simply connected, then for all $p > 1$, we have that
\begin{equation}\label{OssCro1}
\lambda_{1,p}(\Omega) \geq \left( \frac{1}{p \ \tilde{\rho}_{\Omega}}\right)^p.
\end{equation}
If $\Omega$ is of connectivity $k \geq 2$, then for all $p > 1$, we have that
\begin{equation}\label{OssCro2}
\lambda_{1,p}(\Omega) \geq \frac{2^{p/2}}{k^{p/2} p^p \ \rho_{\Omega}^p}.
\end{equation}
\end{theorem}

It is hard to say whether these bounds are optimal for the case $p\neq 2$ since we can not compute eigenfunctions and eigenvalues explicitly on simple domains, unlike the case of the usual Laplacian. Also note that for the Laplace operator, instead of the constant $\frac{1}{4}$ given by \eqref{OssCroP1}, the better constant $\approx 0.6197$ was found using probabilistic methods in \cite{BaC}.

\begin{remark} For a bounded domain $\Omega$ in $\mathbb R^n$, the ground state problem associated to the $\infty$-Laplacian is the following:
\begin{equation}
\min\left\{ |\nabla u| - \lambda_{1,\infty} u; -\Delta_{\infty}u\right\} = 0,
\end{equation}
where $\Delta_\infty u:= \sum_{i,j=1}^n \frac{\partial u}{\partial x_i}\frac{\partial u}{\partial x_j}\frac{\partial^2 u}{\partial x_i \partial x_j}$. It is a notable feature that $\lambda_{1,\infty} = \frac{1}{\rho_\Omega}$, i.e. the value of $\lambda_{1,\infty}$ can immediately be read off the geometry of $\Omega$, without any topological assumptions on $\Omega$ (see \cite{JL} and reference therein for additional details).
\end{remark}

\subsection{The limiting case p=1}

As $p\to1$, the limit equation formally reads
\begin{eqnarray}
-\operatorname{div}\left (\frac{\nabla u}{|\nabla u|}\right) = \lambda_{1,1}(\Omega) \qquad\mbox{ in } \Omega, \\
u=0 \quad\qquad \mbox{ on } \partial\Omega, \nonumber
\end{eqnarray}
where $\lambda_{1,1}(\Omega) := \lim_{p\to1^+} \lambda_{1,p}(\Omega) = h(\Omega),$ where $h(\Omega)= \inf_{D\subset \Omega} \frac{|\partial \Omega|}{|\Omega|}$ with $D$ varying over all non-empty sets $D\subset \Omega$ of finite perimeter (see \cite{KF, KS} for instance). Here, $\Omega$ is assumed to be smooth enough.
If we restrict the subdomains considered in the computation of $h(\Omega)$ to smooth simply connected ones, we get

\begin{proposition}\label{p=1}
If $\Omega$ is a simply connected planar domain, then
\begin{equation}
\lambda_{1,1}(\Omega) \geq \frac{1}{\tilde{\rho}_{\Omega}}.
\end{equation}
\end{proposition}

\subsection{A bound on the size of the nodal set in the planar case}

This subsection is devoted to the study of the size of the nodal set $ \mathcal{Z}_\lambda = \{ x\in \Omega : u_\lambda (x) =0 \}$ of an eigenfunction $u_\lambda$ of \eqref{pLaplacian}. Yau's Conjecture (see \cite{Y}) asserts that the size of the nodal set of an eigenfunction $u_\lambda$ is comparable to $\lambda^{1/2}$.

Donnelly and Fefferman (see \cite{DF}) proved Yau's Conjecture for real analytic manifolds. However, if one assumes only that $(M,g)$ is smooth, Yau's Conjecture remains partially open. In the planar case, it is known that $\mathcal{H}^1(\mathcal{Z}_\lambda) \geq C_1 \lambda^{1/2},$ where $\mathcal{H}^1$ stands for the $1$-dimensional Hausdorff measure (see \cite{Bru}). If $n\geq3$, lower bounds were obtained (see recent works of Sogge-Zelditch in \cite{SZ}, Colding-Minicozzi in \cite{CM}, and Mangoubi in \cite{M}).

We generalize the lower bound in the planar case for the $p$-Laplacian on a planar bounded domain. However, the situation is slightly different from the one for the Laplace operator mainly since it is still not known whether the interior of $\mathcal{Z}_\lambda$ is empty or not (see for instance \cite{D1, GM} for a discussion on that matter). The result is the following:

\begin{theorem}\label{nodalsets}
Let $\Omega$ be a planar bounded domain. There is a constant $M>0$ such that for $\lambda>M$, there exists a positive constant $C$ such that
\begin{equation}
\mathcal{H}^1(\mathcal{Z}_\lambda) \geq C \lambda^{1/p}.
\end{equation}
\end{theorem}

In order to prove the previous theorem, we need to start proving the analog of a classical result of the Dirichlet Laplacian stating that every eigenfunction must vanish in a ball of radius comparable to the wavelength:

\begin{proposition}\label{zeros}
Let $\Omega \subset \mathbb{R}^n$ be a bounded domain.  Let $R=\left( \frac{C}{\lambda}\right)^{\frac{1}{p}},$ where $C>\lambda_{1,p}(B_1),$ the first eigenvalue of a ball of radius $1$. Then, any eigenfunction $u_\lambda$ of \eqref{pLaplacian} vanishes in any ball of radius $R$.
\end{proposition}
This result can be deduced from \eqref{monotonicite}.

%

\section{Lower bounds involving the inradius in higher dimensions}\label{S3}

\subsection{Lieb's and Maz'ya-Shubin's approaches to the inradius problem}

Recall that if $p\leq n-1$, W. K. Hayman's counterexample holds. Therefore, it is not possible to get such lower bounds in that case even if $\Omega$ is simply connected or if $\partial \Omega$ connected. Nevertheless, E. Lieb obtained a similar lower bound by relaxing the condition that the ball has to be completely inscribed in $\Omega$. By doing so, one can also relax some hypotheses on $\Omega$. Instead, throughout this section, we shall only assume that $\Omega$ is an open subset of $\mathbb{R}^n$. No assumptions on the boundedness or on the smoothness of the boundary of $\Omega$ are required. We denote the bottom of the spectrum of $-\Delta_p$ by $\lambda_{1,p}(\Omega)$. In the case that $\Omega$ is a bounded domain, $\lambda_{1,p}(\Omega)$ corresponds to the lowest eigenvalue of $-\Delta_p$ with Dirichlet boundary condition as defined in \eqref{var1}. In the general case, we write
\begin{equation}\label{gvar1}
\lambda_{1,p}(\Omega) = \inf_{u \in C^{\infty}_0(\Omega)}  \left \{  \frac{\int_\Omega |\nabla|^p dx}{\int_\Omega |u|^p dx} \right\}.
\end{equation}
Lieb's method to get such a bound is to allow a fixed fraction $\alpha \in (0,1)$ of the Lebesgue measure of the ball to stick out of $\Omega$. The result can be found in \cite[p. 446]{L}. It states that for a fixed $\alpha \in (0,1)$, if
$$\sigma_{n,p}(\alpha)= \lambda_{1,p}(B_r) \delta_n^{p/n}(\alpha^{-1/n}-1)^p,$$
where $\delta_n = \frac{r^n}{|B_r|}$, then
\begin{equation}\label{Lieb}
\lambda_{1,p}(\Omega) \geq \frac{\sigma_{n,p}(\alpha)}{(r^{L}_{\Omega,\alpha})^p}.
\end{equation}
Here $r^{L}_{\Omega,\alpha}=\sup\{r:\exists B_r\mbox{ such that } |B_r \setminus \Omega| \leq \alpha |B_r|\}$. Notice that the constant in the lower bound is not totally explicit since it depends on $\lambda_{1,p}(B_1)$, which is not known explicitly for $p\neq 2$.

Maz'ya and Shubin obtained a similar bound, but instead of using Lebesgue measure, they considered the Wiener capacity. The goal of this section is to generalize results of \cite{MS} to the $p$-Laplacian to cover the case where $p < n$. To do so, we mainly use results stated in \cite{Maz} and simplify the approach used in \cite{MS}, while losing the explicit constants in the bounds. Recall that $p$-capacity is defined for a compact set $F\subset \Omega$, where $n > 2$, as
\begin{equation}
\operatorname{cap}_p(F)= \inf_u\left \{  \int_{\mathbb{R}^n} |\nabla u|^p: u \in C^\infty_0(\Omega), u|_F \geq 1 \right \}.
\end{equation}
Fix $\gamma \in (0,1)$. A compact set $F \subset B_r$ is said to be $(p,\gamma)$-negligible if
\begin{equation}
\operatorname{cap}_p(F) \leq \gamma \operatorname{cap}_p(\closure {B}_r).
\end{equation}
We are ready to state the main result of this section and its corollaries:
\begin{theorem}\label{thm1}
If $1 < p \leq n$, there exist two positive constants $K_1(\gamma,n,p)$ and $K_2(\gamma,n,p)$ that depend only on $\gamma, n, p$ such that
\begin{equation}
K_1(\gamma,n,p) r_{\Omega,\gamma}^{-p} \leq \lambda_{1,p}(\Omega) \leq K_2(\gamma,n,p) r_{\Omega,\gamma}^{-p},
\end{equation}
where $r_{\Omega,\gamma} = \sup \left \{ r : \exists B_r, \closure{B}_r \setminus \Omega \mbox{ is } (p,\gamma)-\mbox{negligible} \right \}$ is the interior $p$-capacity radius.
\end{theorem}
A direct application of Theorem \ref{thm1} is the following:
\begin{corollary}\label{coro1}
 If $1 < p \leq n$, then $\lambda_{1,p}(\Omega) > 0 \iff r_{\Omega,\gamma} < + \infty$.
\end{corollary}
Corollary \ref{coro1} gives a necessary and sufficient condition of strict positivity of the operator $-\Delta_p$ with Dirichlet boundary conditions. For instance, let $\Omega$ be the complement of any Cartesian grid in $\mathbb{R}^{n\geq3}$ (this example is adapted from \cite[p. 789]{Maz}. Clearly, for any $\gamma \in (0,1)$, $r_{\Omega,\gamma}= +\infty$; thus, $\lambda_{1,p}(\Omega) =0$. However, if $\Omega$ is a narrow strip, $r_{\Omega,\gamma} <  +\infty$, implying that $ \lambda_{1,p}(\Omega)>0$.

 Note that since $\lambda_{1,p}(\Omega)>0$ does not depend on $\gamma$, we immediately get the following:
\begin{corollary}
 If $1 < p \leq n$, the conditions $ r_{\Omega,\gamma} < + \infty$ for different $\gamma$'s are all equivalent.
\end{corollary}
Also, one can show that for $1 < p \leq n$, the lower bound given by Theorem \ref{thm1} implies the lower bound obtained earlier by Lieb, \eqref{Lieb}. To do so, one needs to use an isocapacity inequality that can be found in \cite[Section 2.2.3]{Maz}, stating that if $F$ be a compact subset of $\mathbb{R}^n$, then
\begin{equation}\label{mazeq}
\operatorname{cap}_p(F) \geq \omega^{p/n} n^{(n-p)/n}\left(\frac{n-p}{p-1}\right) |F|^{(n-p)/n},
\end{equation}
where equality occurs if and only if $F$ is a ball.
\begin{proposition}
If $\alpha = \gamma^{n/(n-p)}$ and $1<p\leq n$, then $r^{L}_{\Omega,\alpha} \geq r_{\Omega,\gamma}$. In particular, Theorem \ref{thm1} implies \eqref{Lieb}.
\end{proposition}
\begin{proof}
Let $C =  \omega^{p/n} n^{(n-p)/n}\left(\frac{n-p}{p-1}\right)$ and fix $\gamma \in ] 0,1 [$. Suppose that
$$ \operatorname{cap}_p( \closure{B}_r \setminus\Omega) \leq \gamma \operatorname{cap}_p(\closure{B}_r);$$
therefore, using \eqref{mazeq}, we get that
\begin{eqnarray*}
|\closure{B}_r \setminus \Omega| & \leq & C^{-\frac{n}{n-p}} \operatorname{cap}_p(\closure{B}_r \setminus \Omega) \\
& \leq & C^{-\frac{n}{n-p}} \gamma^{\frac{n}{n-p}}\operatorname{cap}_p(\closure{B}_r)^{\frac{n}{n-p}} \\
& = &  \alpha  |\closure{B}_r|,
\end{eqnarray*}
yielding the desired result.
\end{proof}

\subsection{Convex domains in $\mathbb R^n$}

Another way to avoid such difficulty is to consider convex domains in $\mathbb R^n$. Doing so, we can prove the following:

\begin{proposition}
If $\Omega$ be a convex body in $\mathbb R^n$, then the following inequality
\begin{equation}
\lambda_{1,p}(\Omega) \geq \left( \frac{1}{p\rho_\Omega} \right)^p,
\end{equation}
holds for all $p>1$:
\end{proposition}

The proof is based on two key facts (see \cite[p. 26]{O3}), namely that the inequality,
$$ |\partial \Omega | \geq h(\Omega) |\Omega|,$$
is required to be true only for subdomains bounded by level surfaces of the first eigenfunction of $\Omega$ (see the proof of Lemma \ref{CheegerLemma}), and that if $\Omega$ is convex, then those subdomains $\Omega'$ are also convex (see \cite[Theorem 1.13]{BL}). Recalling that $\frac{1}{\rho_{\Omega'}} \geq \frac{1}{\rho_\Omega}$, the proposition then follows easily.

\subsection{Further discussion}

A striking difference between the usual Laplace operator and the $p$-Laplacian is that it is possible to obtain bounds of the type \eqref{goal} in higher dimensions, as long as $p$ is "large enough" compared to the dimension. Indeed, W.K. Hayman's observation remains valid in the case $p \leq n-1$ since for such $p$, every curve has a  trivial $p$-capacity. On the other hand, for $p>n$, even a single point has a positive $p$-capacity (see \cite[Chapter 13, Proposition 5 and its corollary]{Maz}). Consequently, W. K. Hayman's counterexample no longer holds since removing a single point has an impact on the eigenvalues as it can be seen from \eqref{var1}. Taking the latter observation into account, it is possible to prove the following theorem:

\begin{theorem}\label{ThmPrincipal}
Let $p > n$ and let $\Omega$ be a bounded domain. Then there exists a positive constant $C_{n,p}$ such that
\begin{equation}\label{EqPrincipal}
\lambda_{1,p}(\Omega) \geq \dfrac{C_{n,p}}{\rho_{\Omega}^{p}}.
\end{equation}
\end{theorem}

It is known that if $p > n-1$, every curve has a positive $p$-capacity. Therefore, W. K. Hayman's counter example does not work in that case, leading to the following result:



\begin{theorem} \label{CurPosCap}
For $ p > n-1$, suppose that $\Omega$ is a bounded domain such that $\partial \Omega$ is connected, then there exists $C_{n,p}>0$ such that
$$\lambda_{1,p}(\Omega) \geq \frac{C_{n,p}}{\rho_\Omega^{p}}.$$
\end{theorem}

This result can be viewed as a generalization of classical results for the Laplacian on simply connected planar domains ($p=n=2$) discussed earlier however, without the explicit constant. Also notice that the result can not hold if $p=n-1$ as noted by W. K. Hayman for the case $p=2, n=3$.

Theorems \ref{ThmPrincipal} and \ref{CurPosCap} were suggested by D. Bucur and their proofs can be found in \cite{Poli}.

\section{Lower bounds involving eigenvalues of the $p$-Laplace operator on manifolds}\label{S4}


\subsection{A two-sided inradius bound for nodal domains on closed Riemannian surfaces}

Some results concerning the inner radius of nodal domains on manifolds were obtained in \cite{M1, M2}. As suggested by D. Mangoubi, it is possible to extend a two sided estimate valid for closed surfaces, \cite[Theorem 1.2]{M1} to the case of the $p$-Laplacian. Let $U_{\lambda}$ denote the $\lambda$-nodal domain associated to the eigenfunction $u_{p,\lambda}$. The result is the following:
\begin{theorem}\label{Manthm}
Let $(S,g)$ be a closed surface. Then, there exists two positive constants $c, C$ such that
$$c \lambda^{-1/p} \leq \rho_{U_{\lambda}} \leq C \lambda^{-1/p}.$$
\end{theorem}
The proof of the case $p=2$ relies on two main tools, namely Faber-Krahn, which still holds for the $p$-Laplace operator, and a Poincaré inequality, \cite[Theorem 2.4]{M1}. The latter can also be generalized to our setting. We have the following result:
\begin{lemma}\label{PoinCare}
Let $Q\subseteq \mathbb{R}^2$ be a cube whose edge is of length $a$. Fix $\alpha \in (0,1)$ and let $u$ be in $C^\infty_0(Q)$ such that $u$ vanishes on a curve whose projection on one of the edges of Q has size $\geq \alpha a.$ Then, there exists a constant $C(\alpha)$ such that
$$ \int_Q |u|^p dx \leq C(\alpha) a^p \int_Q |\nabla u|^p dx.$$
\end{lemma}
\noindent Since the proof of  Theorem \ref{Manthm} and Lemma \ref{PoinCare} are direct extensions of what is done in \cite{M1}, we omit details.


\subsection{Lower bounds involving the inradius on surfaces}

We can adapt some inradius results of \cite{O1} valid for simply connected domains of surfaces $(S,g)$ with controlled Gaussian curvature to the case of the $p$-Laplacian:
\begin{proposition}\label{OssCroP3}
Let $S$ be a surface. Let $D \subset S$ be a simply connected domain. Denote by $K$ its Gaussian curvature and by $\beta$ the quantity $\int_D K^+,$ where $K^+=\max\{K,0\}$. If $\beta \leq 2\pi$ the inequality,
\begin{equation}\label{OssCro31}
\lambda_{1,p}(D) \geq \left( \frac{1}{p \ \rho_{D}}\right)^p,
\end{equation}
holds.
\end{proposition}
\noindent If the surface has a negative Gaussian curvature, then we have the following:
\begin{proposition}\label{OssCroP4}
Let $S$ be a simply connected surface. Let $D \subset S$ be a simply connected domain such that $K \leq -\alpha^2$ on $D, \ \alpha > 0,$  where $K$ stands for its Gaussian curvature.
The stronger inequality
\begin{equation}\label{OssCro4}
\lambda_{1,p}(D) \geq \left(\frac{\alpha \coth(\alpha \rho_D)}{p}\right)^p
\end{equation}
holds.

\noindent Furthermore, if $S$ is a complete surface, then one has the following,
\begin{equation}\label{OssCro5}
\lambda_{1,p}(D) \geq \left(\frac{\alpha \coth(\alpha R_D)}{p}\right)^p,
\end{equation}
where $R_D$ stands for the circumradius of $D$.
\end{proposition}

\subsection{Negatively curved manifolds}

The first result, called McKean's theorem (see \cite{McK}) in the case of the Laplace operator, concerns manifolds of negative sectional curvature:
\begin{proposition}\label{McKean}
Let $(M,g)$ be a complete and simply connected Riemannian manifold. Let $D\subset M$ be a domain such that its sectional curvature is bounded by $\leq -\alpha^2, \alpha>0$ , then
\begin{equation}
\lambda_{1,p}(D) \geq \frac{(n-1)^p\alpha^p}{p^p}.
\end{equation}
\end{proposition}

The next result is valid on an arbitrary surface (without additional assumptions, such as being simply connected), but it is only valid for  doubly-connected domains.
\begin{proposition}\label{Prop1}
Let $(S,g)$ be a surface. Let $D\subset S$ be a doubly-connected domain such that $K\leq -1$ where $K$ stands for its Gaussian curvature, then
\begin{equation}
\lambda_{1,p}(D) \geq \frac{1}{p^p}.
\end{equation}
\end{proposition}


\subsection{Minimal submanifolds in $\mathbb R^n$}

It is also possible to prove the following using similar arguments:
\begin{proposition}\label{Prop2}
Let $D$ be a domain on a $m$-dimensional minimal submanifold in $\mathbb R^n$. If $D$ lies in a ball of radius $R$, then
\begin{equation}
\lambda_{1,p}(D)\geq \left( \frac{m}{p} \right)^p.
\end{equation}
\end{proposition}

\section{Proofs} \label{S5}

\subsection{Proof of Lemma \ref{CheegerLemma}}

By \eqref{chegconst}, it follows that if one proves \eqref{14} for all domains in a regular exhaustion of $D$, then \eqref{14} will also hold for $D$. Knowing that every finitely-connected domain has a regular exhaustion by domains of the same connectivity, we may assume that $D$ has a smooth boundary. Hence, the $p$-Laplacian Dirichlet eigenvalue problem admits a solution $u_1$ corresponding to $\lambda_{1,p}$ and may be chosen without loss of generality such that $u_1 \geq 0$.

Let $g=u_1^p$. Then, Hölder's inequality implies that
\begin{eqnarray*}
\int_{D} |\nabla g(x)| dx &=& p \int_{D} |u_1(x)|^{p-1}|\nabla u_1(x)| dx \\
& \leq & p ||u_1||^{p-1}_p ||\nabla u_1||_p.
\end{eqnarray*}
Dividing by $||u_1||^p_p$, one gets
\begin{eqnarray}\label{test1}
\frac{1}{p} \frac{\int_{D} |\nabla g(x)| dx }{\int_{D} |g(x)| dx } \leq \frac{||\nabla u_1||_p}{|| u_1||_p} = \lambda_{1,p}(D)^{1/p}.
\end{eqnarray}

For regular values $t$ of the function $g$, we define the set
$$ D_t = \left\{ y\in D: g(y) > t \right\}.$$
We want to show that the connectivity of $D_t$ is at most $k$, since it will imply by \eqref{chegconst} that
\begin{equation} L(t) \geq h_k(D) A(t), \end{equation}
for all regular values of $t$. Here, $L(t)$ is the length of the boundary of $D_t$ and $A(t)$ the area of $D_t$.

Since $t$ is regular, the boundary $\partial D_t$ consists of a finite number, say $m$, of smooth curves $C_1, C_2, ..., C_m$ along which $\nabla g \neq 0$. Let $D_t'$ be any connected component of $D_t$. If the connectivity of $D_t'$ were greater than $k$, then the complement of $D_t$ would contain a component lying completely in $D$ of boundary say, $C_l$. Since $u_1 \in C(D)\cap W^{1,p}(D)$ and
$$ \int_{D} |\nabla u_1|^{p-2} \nabla u_1 \cdot \nabla v \ dx = \lambda_{1,p} \int_D |u_1|^{p-2} u_1  v \ dx \geq 0,$$
for all non negative $v$ in $C_0^{\infty}(D)$, $u_1$ is $p$-superharmonic (see \cite[Theorem 5.2]{Lind3}). By definition of $p$-superharmonicity (again \cite[Definition 5.1]{Lind3}), the comparison principle holds. Therefore, since $u_1 = \sqrt[p]{t}$ on $C_l$, it follows that $u_1^p \geq t$ in the internal region of $C_l$, which contradicts the fact that the internal region lies in the complement of $D_t'$, so that the function $g$ has to be such that $g < t$.

Now, since the set of singular values of $g$ is a closed set of measure zero by Sard's theorem, its complement is a countable union of open intervals $I_n$. Thus, we can define the following
$$E_n = \left\{ p \in D : g(p) \in I_n \right\}.$$
Then, by the coarea formula,
\begin{eqnarray}\label{test2}
\int_D |\nabla g(x)| \ dx & \geq & \sum_{n=1}^\infty \int_{E_n} |\nabla g(x)| \ dx = \sum_{n=1}^\infty \int_{I_n} L(t) \ dt \nonumber \\
& \geq & h_k(D)  \sum_{n=1}^\infty  \int_{I_n}  A(t) \ dt = h_k(D) \int_D g(x) \ dx.
\end{eqnarray}
Combining \eqref{test1} with \eqref{test2} completes the proof.

\subsection{Proof of Theorem \ref{OssCroP1}, of Propositions \ref{OssCroP3}, and \ref{OssCroP4}}\label{P1}

We use the following lemma which was proved in \cite{G}:
\begin{lemma}\label{Grieser}
If $\Omega' \subset \Omega$, then
$$\tilde{\rho}_{\Omega' } \leq \tilde{\rho}_\Omega.$$
\end{lemma}

\begin{proof}[Proof of Theorem \ref{OssCroP1}]

Combining Lemma \ref{CheegerLemma} with Lem\-ma \ref{Grieser} and Bonnesen's inequality (see \cite{Bon}) then yields the desired result.


In order to adapt the argument to planar domains of connectivity $k$, one must use a generalized geometric inequality that can be found in \cite[Theorem 1]{Cr}, stating that
$$\frac{|\partial \Omega|}{|\Omega|} \geq \frac{\sqrt{2}}{\sqrt{k}\rho}.$$
Combining Lemma \ref{CheegerLemma}, \cite[Theorem 1]{Cr}, and the fact that  $\rho_{\Omega'} \leq \rho_\Omega$ provided that $\Omega' \subset \Omega$ yields the desired result.
\end{proof}

\begin{proof}[Proof of Proposition \ref{OssCroP3}]

We start by using Burago and Zalgaller's inequality that can be found in \cite{BZ}, and that states that
\begin{equation}\label{BZ} \rho_D |\partial D| \geq |D| + \left( \pi - \frac{1}{2} \beta  \right) \rho_D^2 \quad \iff \quad \frac{|\partial D|}{|D|} \geq \frac{1 + \frac{(\pi - \frac{1}{2}\beta) \rho_D^2}{|D|}}{\rho_D} \geq \frac{1}{\rho_{D}}.\end{equation}
By definition of $h_1(D)$, together with \eqref{BZ}, one gets the desired result.

\end{proof}

\begin{proof}[Proof of Proposition \ref{OssCroP4}]
Since $S$ is simply connected, one can use \cite[Theorem 1]{I} or  \cite[Theorem 8(c)]{O3} to get that
$$|\partial D| \geq \frac{\alpha |D|}{\tanh \alpha\rho_D}  + \frac{2\pi}{\alpha} \tanh \frac{\alpha \rho_D}{2} \geq \alpha |D| \coth( \alpha\rho_D).$$
Combining the previous inequality with Lemma \ref{CheegerLemma} and the fact that  $\rho_{\Omega'} \leq \rho_\Omega$ provided that $\Omega' \subset \Omega$ proves the first part of Theorem \ref{OssCroP4}. To conclude the proof, one must use \cite[Theorem 8]{O3}, which states that
$$|\partial D| \geq \alpha |D| \coth( \alpha R_D).$$
Combining the latter inequality with Lemma \ref{CheegerLemma} and the fact that for any subdomain $D'$, its circumradius satisfies $R' \leq R, \coth(\alpha R') \geq \coth(\alpha R)$ yields the desired result.
\end{proof}

\subsection{Proof of Proposition \ref{zeros}}

Consider a ball $B_R$ and suppose that $u_\lambda\neq 0$ inside $B_R$. Since $u_\lambda \neq 0$ in $B_R$, there exists a nodal domain $A$ of $u_\lambda$ such that
$B_R \subset A.$
By  \eqref{monotonicite}, we have that
$$R\leq\rho_A\leq \left(\frac{\lambda_{1,p}(B_1)}{\lambda_{1,p}(A)}\right)^{\frac{1}{p}}.$$
Since the restriction of $u_\lambda$ corresponds to the first eigenfunction on $A, \lambda_{1,p}(A) = \lambda $. Thus, we get that
$$R\leq \left(\frac{\lambda_{1,p}(B_1)}{\lambda}\right)^{\frac{1}{p}},$$
leading to a contradiction.

\subsection{Proof of Theorem \ref{nodalsets}}

Notice that for the $p$-Laplacian, it is still unclear whether $\operatorname{int}(\mathcal{Z}_\lambda)$ is empty or not. Nevertheless, if $\operatorname{int}(\mathcal{Z}_\lambda)\neq \emptyset$, then $\mathcal{H}^1(\mathcal{Z}_\lambda) = + \infty$.

Suppose that $\operatorname{int}(\mathcal{Z}_\lambda)= \emptyset$. We need to show that $\exists C>0$ such that $\mathcal{H}^1(\mathcal{Z}_\lambda) \leq C \lambda^{\frac{1}{p}}$ for $\lambda$ large enough.

\begin{figure}[htb]
\begin{center}
\subfigure[Case 3]{
\includegraphics[height=1.5in,width=1.5in,angle=-90]{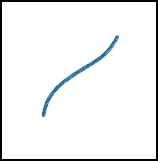}
\label{fig:cas3}
}
\subfigure[Case 4]{
\includegraphics[height=1.5in,width=1.5in,angle=-90]{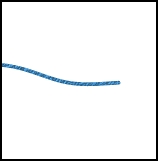}
\label{fig:cas4}
}
\caption{Nodal lines in a square.}
\end{center}
\end{figure}


%
%

By Proposition \ref{zeros}, $\Omega$ can be split into squares $S_c$ of side $c=\operatorname{Area}(\Omega)\lambda^{-1/p}$ such that each square contains a zero of $u_\lambda$. Take $\lambda$ large enough to allow that the center of each square corresponds to a zero of $u_\lambda$. We represent the various cases of nodal lines in a square in Figure 1 and Figure 2.  Recall the following Harnack inequality:

\begin{theorem}[Theorem 1.1 of \cite{Tr}]\label{Harnack} Let $K = K(3\rho) \subset \Omega$ be a cube of length $3\rho$. Let $u_\lambda$ be a solution of \eqref{pLaplacian} associated to the eigenvalue $\lambda$ such that $u_\lambda(x) \geq 0$ for all $x \in K$, then
$$ \max_{K(\rho)} u_\lambda(x) \leq C \min_{K(\rho)} u_\lambda(x),  $$
where $C$ is a positive constant that depends on $n,p$ and on $\lambda$.
\end{theorem}

Notice that in order for Theorem \ref{Harnack} to fail, we know that every neighborhood of the boundary of the nodal domain must contain points such that $u_\lambda$ changes sign. To do so, the nodal line must be closed. Therefore, nodal lines depicted in Figure 1 can not occur.

\begin{figure}[htb]
\begin{center}
\subfigure[Case 1]{
\includegraphics[height=1.5in,width=1.5in,angle=-90]{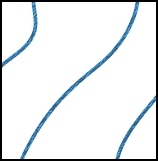}
\label{fig:cas1}
}
\subfigure[Case 2]{
\includegraphics[height=1.5in,width=1.5in,angle=-90]{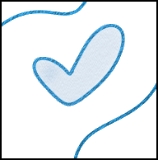}
\label{fig:cas2}
}
\caption{Closed nodal lines in a square.}
\end{center}
\end{figure}

Suppose that there is a closed nodal line inside a square. If it were the case, it would mean that there exists a nodal domain $A$ included the square. Since the eigenfunction restricted to $A$ would then correspond to the first one, domain monotonicity would yield a contradiction. Indeed, we would have the following
$$ \lambda_{1,p}(S_1) \lambda \operatorname{Area}(\Omega)^{-p} = \lambda_{1,p} (S_c) \leq \lambda_{1,p}(A) = \lambda,$$
yielding that
$$ \lambda_{1,p}(S_1) \leq \operatorname{Area}(\Omega)^{p},$$
a contradiction (simply rescale $\Omega$ if necessary).

Therefore, any nodal line inside any such square must be at least of length greater than or equal to $C_1 \lambda^{-1/p}$ . Since $\operatorname{Area}(\Omega)= \operatorname{Area}(S_c) \lambda^{2/p}$, there are roughly $ \lambda^{2/p}$ squares covering $\Omega$, implying that

$$\mathcal{H}^1(\mathcal{Z}_\lambda)\geq C_2 \lambda^{2/p} \lambda^{-1/p} = C_2 \lambda^{1/p}.$$


\subsection{Proof of Theorem \ref{thm1}}

Using \cite[Theorem 14.1.2]{Maz},
we get the following lemma:
\begin{lemma}\label{L1}
 Let $F$ be a compact subset of $\closure{B}_r$.

\noindent 1. If $1 < p \leq n$, for all $u \in C^{\infty}(\closure{B}_r)$ such that $u \equiv 0$ on $F$,  there exists a positive constant $C_1(n,p)$ depending only on $n$ and $p$ such that
\begin{equation}
\operatorname{cap}_p(F) \leq \frac{C_1(n,p)}{r^{-n}} \frac{\int_{\closure{B}_r} |\nabla u|^p}{\int_{\closure{B}_r} |u|^p}.
\end{equation}

\noindent 2. If $1<p \leq n$, for all $u \in C^{\infty}(\closure{B}_r)$ such that $u \equiv 0$ on $F$, where $F$ is a negligible subset of $\closure{B}_r$, and
\begin{equation}
||u||_{L^p(\closure{B}_{r/2})} \leq C || \nabla u ||_{L^p(\closure{B}_{r})},
\end{equation}
then there exists a positive constant $C_2(n,p)$ depending only on $n$ and $p$ such that
\begin{equation}
\operatorname{cap}_p(F) \geq \frac{C_2(n,p)}{r^{-n}} \frac{\int_{\closure{B}_r} |\nabla u|^p}{\int_{\closure{B}_r} |u|^p}.
\end{equation}
\end{lemma}
 Let us also recall the following two properties of the $p$-capacity (which are proved in \cite{Maz}):
\begin{itemize}
\item The $p$-capacity is monotone: $F_1 \subset F_2 \implies \operatorname{cap}_p (F_1) \leq \operatorname{cap}_p(F_2)$.

\item The $p$-capacity of a closed ball of radius $r$ can be computed explicitly:
\begin{equation}
\operatorname{cap}_p(\closure{B}_r) = r^{n-p}\operatorname{cap}_p(\closure{B}_1) = r^{n-p} \omega_n \left ( \frac{|n-p|}{p-1}\right)^{p-1},
\end{equation}
where $\omega_n$ is the area of the unit sphere $S^{n-1} \subset \mathbb{R}^n$ and $p \neq n$.
\end{itemize}
We are ready to begin the proof of the lower bound. The ideas used in the following proof come from \cite{MS}.
\begin{proof}[Lower bound of Theorem \ref{thm1}]

Fix $\gamma \in (0,1)$ and choose any $r > r_{\Omega,\gamma}$. Then, any ball $\closure{B}_r$ is of non-negligible intersection, i.e.
\begin{equation} \label{nonnegligible}
\operatorname{cap}_p(\closure{B}_r \setminus \Omega) \geq \gamma \operatorname{cap}_p(\closure{B}_r). \nonumber
\end{equation}
Since any $u \in C_0^\infty(\Omega)$ vanishes on $\closure{B}_r \setminus \Omega$, one can use Lemma \ref{L1}, part 1. Using \eqref{nonnegligible} and the explicit value of the $p$-capacity of a closed ball, one gets the following:
\begin{eqnarray}\label{step1}
\int_{\closure{B}_r} |u|^p dx & \leq & \frac{C_1(n,p)}{r^{-n}  \operatorname{cap}_p(\closure{B}_r \setminus \Omega)} \int_{\closure{B}_r} |\nabla u|^p dx \\
& \leq & \frac{C_1(n,p)}{r^{-n} \gamma  \operatorname{cap}_p(\closure{B}_r)} \int_{\closure{B}_r} |\nabla u|^p dx \nonumber \\
& \leq & \frac{C_1(n,p)}{r^{-p} \gamma  \operatorname{cap}_p(\closure{B}_1)} \int_{\closure{B}_r} |\nabla u|^p dx. \nonumber
\end{eqnarray}

Choose a covering of $\mathbb{R}^n$ by balls $\closure{B}_r = \closure{B}_r^{(k)}$, $k=1,2, ...$, so that the multiplicity of this covering is at most $N = N(n)$, which is bounded since
for example, for $n\geq 2$, the following estimate is valid (see \cite[Theorem 3.2]{R}):
\begin{equation}
N(n) \leq n \log(n) + n \log(\log(n)) + 5n. \nonumber
\end{equation}
Sum up \eqref{step1} to get the following:
\begin{eqnarray*}
\int_{\mathbb{R}^n} |u|^p dx & \leq & \sum_k \int_{\closure{B}^{(k)}_r}  |u|^p dx \\
& \leq  & \frac{C_1(n,p)}{r^{-p}\gamma  \operatorname{cap}_p(\closure{B}_1)}  \sum_k \int_{\closure{B}^{(k)}_r} |\nabla u|^p dx \\
& \leq  & \frac{C_1(n,p) N(n)}{r^{-p}\gamma  \operatorname{cap}_p(\closure{B}_1)} \int_{\mathbb{R}^n} |\nabla u|^p dx .
\end{eqnarray*}
Since for all $u \in C_0^\infty(\Omega)$, we have that
\begin{equation}
\frac{\gamma \operatorname{cap}_p(\closure{B}_1)r^{-p}}{C_n  N(n)} \leq \frac{ \int_{\Omega} |\nabla u|^p dx}{ \int_{\Omega} |u|^p dx }, \nonumber
\end{equation}
we get that
\begin{equation}\label{eqfin}
\lambda_{1,p}(\Omega) \geq K_1(\gamma, n ,p) r^{-p} =  \frac{\gamma  \operatorname{cap}_p (\closure{B_1})}{C_1(n,p)  N(n)} r^{-p}.
\end{equation}
Taking the limit of \eqref{eqfin} as $r \searrow r_{\Omega, \gamma}$ yields the desired result.
\end{proof}

The proof of the upper bound is very similar to the last one, but uses the second part of Lemma \ref{L1}. However, this proof is different from the one given in \cite{MS}, but has the disadvantage of not yielding an explicit constant. Nevertheless, no such constant are known in the case of the $p$-Laplacian (recall that Lieb's constant for the lower bound and that the upper bound given in \eqref{monotonicite} are not totally explicit since they both depend on $\lambda_{1,p}(B)$).

\begin{proof}[Upper bound of Theorem \ref{thm1}]

Fix $\gamma \in (0,1)$ . Consider $r_{\Omega,\gamma}$. By definition, we know that
\begin{equation}
\operatorname{cap}_p(\closure{B}_{r_{\Omega,\gamma}} \setminus \Omega) \leq \gamma \operatorname{cap}_p(\closure{B}_{r_{\Omega,\gamma}}). \nonumber
\end{equation}
We know want to use Lemma \ref{L1}, part $2$. Let $F=\closure{B}_{r_{\Omega,\gamma}} \setminus \Omega$. Clearly, $F$ is a negligible subset of $\closure{B}_{r_{\Omega,\gamma}}$.
It is also clear that any test function $u \in C_0^\infty(\Omega)$ will vanish identically on $F$. Therefore, for any such function, using Poincaré inequality, we get
\begin{eqnarray*}
||u||_{L^p(\closure{B}_{r_{\Omega,\gamma}/2})} \leq ||u||_{L^p(\closure{B}_{r_{\Omega,\gamma}})} \leq C ||\nabla u||_{L^p(\closure{B}_{r_{\Omega,\gamma}})}.
\end{eqnarray*}
Therefore, one can use Lemma \ref{L1} part $2$ and get :
\begin{eqnarray*}\
\int_{\closure{B}_{r_{\Omega,\gamma}}} |u|^p dx & \geq & \frac{C_2(n,p)}{r_{\Omega,\gamma}^{-n}  \operatorname{cap}_p(\closure{B}_{r_{\Omega,\gamma}} \setminus \Omega)} \int_{\closure{B}_{r_{\Omega,\gamma}}} |\nabla u|^p dx \\
& \geq & \frac{C_2(n,p)}{r_{\Omega,\gamma}^{-n} \gamma  \operatorname{cap}_p(\closure{B}_{r_{\Omega,\gamma}})} \int_{\closure{B}_{r_{\Omega,\gamma}}} |\nabla u|^p dx \nonumber \\
& \geq & \frac{C_2(n,p)}{r_{\Omega,\gamma}^{-p} \gamma  \operatorname{cap}_p(\closure{B}_1)} \int_{\closure{B}_{r_{\Omega,\gamma}}} |\nabla u|^p dx. \nonumber
\end{eqnarray*}

Choose a covering of $\mathbb{R}^n$ by balls $\closure{B}_{r_{\Omega,\gamma}} = \closure{B}_{r_{\Omega,\gamma}}^{(k)}$, $k=1,2, ...$, so that the multiplicity of this covering is at most $N = N(n)$, and get
\begin{eqnarray*}
\int_{\mathbb{R}^n} |\nabla u|^p dx & \leq & \sum_k \int_{\closure{B}^{(k)}_{r_{\Omega,\gamma}}}  |\nabla u|^p dx \\
& \leq  & \frac{r_{\Omega,\gamma}^{-p}\gamma  \operatorname{cap}_p(\closure{B}_1)}{C_2(n,p)}  \sum_k \int_{\closure{B}^{(k)}_{r_{\Omega,\gamma}}} | u|^p dx \\
& \leq  & \frac{r_{\Omega,\gamma}^{-p}\gamma  \operatorname{cap}_p(\closure{B}_1) N(n)}{C_2(n,p)} \int_{\mathbb{R}^n} | u|^p dx .
\end{eqnarray*}
For such $u \in C_0^\infty(\Omega)$, we have that
\begin{equation}
\lambda_{1,p}(\Omega) \leq \frac{\int_{\mathbb{R}^n} |\nabla u|^p dx }{\int_{\mathbb{R}^n} | u|^p dx} \leq \frac{r_{\Omega,\gamma}^{-p} \gamma  \operatorname{cap}_p(\closure{B}_1) N(n)}{C_2(n,p)}; \nonumber
\end{equation}
thus, yielding the desired result with $K_2(\gamma, n,p)= \frac{\gamma  \operatorname{cap}_p(\closure{B}_1) N(n)}{C_2(n,p)}$.
\end{proof}

\subsection{Proofs of Propositions \ref{McKean}, \ref{Prop1}, and \ref{Prop2}}

The proofs are very straightforward and consist of combining Lemma \ref{CheegerLemma} with use \cite[p.26, eq. (121)-(122)]{O3}, \cite[p. 1208, Eq. (4.31)]{O2}, or \cite[Eq. (14)]{O4} respectively.

\proof[Acknowledgements]
This paper is part of the author's Ph.D. thesis under the supervision of Iosif Polterovich. The author is thankful to Iosif Polterovich and to Dima Jakobson for suggesting the problem. Moreover, he wishes to thank Dorin Bucur, Yaiza Canzani, Daniel Grieser, Dan Mangoubi, Pavel Drabek, and Robert G. Owens for useful discussions and references. Finally, the author thanks Gabrielle Poliquin for her help with the figures.

\bibliographystyle{amsplain}

\end{document}